\documentclass[11pt]{amsart}
\usepackage{amsmath,amsthm,amscd,amssymb}
\usepackage{geometry}
\usepackage{pb-diagram}
\usepackage{graphicx}

\usepackage{amssymb}

\newtheorem{thm}{Theorem}[section]
\newtheorem{lem}[thm]{Lemma}
\newtheorem{prop}[thm]{Proposition}
\newtheorem{cor}[thm]{Corollary}

\theoremstyle{definition}

\newtheorem{defn}[thm]{Definition}

\numberwithin{equation}{section}
\numberwithin{figure}{section}

\newcommand{\Z}{\mathbb{Z}}
\newcommand{\K}{\mathcal{K}}
\newcommand{\C}{\mathcal{C}}

\newcommand{\cL}{\mathcal{L}}

\newcommand{\R}{\mathcal{R}}

\usepackage{pstricks}
\usepackage{pst-node}

\title[Injectivity of satellite operators in knot concordance]{Injectivity of satellite operators in knot concordance}

\author{Tim D. Cochran$^{\dag}$}
\address{Department of Mathematics MS-136, P.O. Box 1892, Rice University, Houston, TX 77251-1892}
\email{cochran@rice.edu}

\author{Christopher William Davis}
\address{Department of Mathematics,Hibbard Humanities Hall 508, University of Wisconsin-Eau Claire, WI 54702-4004}
\email{daviscw@uwec.edu}

\author{Arunima Ray}
\address{Department of Mathematics MS-136, P.O. Box 1892, Rice University, Houston, TX 77251-1892}
\email{arunima.ray@rice.edu}

\thanks{$^{\dag}$Partially supported by the National Science Foundation  DMS-1006908}


\begin{document}
\date{\today}
\begin{abstract}  Let $P$ be a knot in a solid torus, $K$ a knot in $S^3$ and $P(K)$ the satellite knot of $K$ with pattern $P$.  This defines an operator $P:\mathcal{K}\to\mathcal{K}$ on the set of knot types and induces a satellite operator $P:\mathcal{C}\to \mathcal{C}$ on the set of smooth concordance classes of knots. There has been considerable interest in whether certain such functions are injective. For example, it is a famous open problem whether the Whitehead double operator is \textbf{weakly injective} (an operator is called weakly injective if $P(K)=P(0)$ implies $K=0$ where $0$ is the class of the trivial knot). We prove that, modulo the smooth  $4$-dimensional Poincar\'{e} Conjecture, any \textbf{strong winding number one} satellite operator is injective on $\C$. More precisely, if $P$ has strong winding number one and $P(K)=P(J)$ then $K$ is smoothly concordant to $J$ in $S^3\times [0,1]$ equipped with a possibly exotic smooth structure. We also prove that any strong winding number one operator is injective on the topological knot concordance group. If $P(0)$ is unknotted then strong winding number one is the same as (ordinary) winding number one. More generally we show that \textit{any}  satellite operator with non-zero winding number $n$ induces an injective function on the set of $\Z[\frac{1}{n}]$-concordance classes of knots. We deduce some analogous results for links.
\end{abstract}

\maketitle

\section{Introduction}\label{sec:intro}
The satellite construction is a classical procedure that transforms an oriented knot $K$ in $S^3$ to another knot.  Suppose $P$ is an oriented knot  in the  solid torus $ST\equiv S^1\times D^2$, called a \textit{pattern knot}. An example is shown in Figure~\ref{fig:satelliteexample}. For any oriented knot  $K$ in $S^3$ we denote by $P(K)$ the (untwisted) \textbf{satellite} of  $K$ obtained by using $P$ as a \textbf{pattern}  ~\cite[p. 10]{Lick2}.  Precise definitions are given in Section~\ref{sec:satellitesandwinding}. Each such pattern may thus be viewed as a function $P:\mathcal{K}\to\mathcal{K}$ on the set of isotopy classes of knots. These descend to  yield functions, called \textbf{satellite operators}, on $\mathcal{K}/\sim$ for various other important equivalence relations, in particular  on the set of concordance classes of knots. We will establish the injectivity of these functions in some important cases.

The importance of satellite operations extends far beyond knot theory. Such operations have been generalized to operations on $3$ and $4$-manifolds where they produce very subtle variations while fixing the homology type ~\cite[Sec. 5.1]{Ha2}. In particular winding number one satellites are closely related to Mazur $4$-manifolds ~\cite{AkKi3} which in turn are closely related to Akbulut \textit{corks}. The latter are contractible $4$-manifolds that can be used to alter the smooth structure on $4$-manifolds (by removing them and reinserting them with a twist). Specifically, a knot $K$ may occur as the attaching circle of a $2$-handle in the handlebody description of a $4$-manifold. It was shown, for example, in ~\cite{AkYa1} that, for the simplest strong winding number one operators $P$, the modification of the handlebody effected by $K\rightsquigarrow P(K)$ can alter the smooth structure on the $4$-manifold without altering the homeomorphism type !

We will, in fact, consider four different ``concordance'' equivalence relations on $\mathcal{K}$, with the sets of equivalence classes being denoted $\C$, $\C^{ex}$, $\C^{top}$, and $\C^{\frac{1}{n}}$ respectively. Here $\C$ denotes the (usual) set of \textbf{smooth knot concordance} classes wherein $K_0\hookrightarrow S^3\times\{0\}$ is equivalent to $K_1\hookrightarrow S^3\times\{1\}$ if there exists a properly, smoothly embedded annulus in $S^3\times [0,1]$ which restricts on its boundary to the given knots. $\C^{top}$ is the (usual) set of \textbf{topological knot concordance} classes wherein $K_0\hookrightarrow S^3\times\{0\}$ is equivalent to $K_1\hookrightarrow S^3\times\{1\}$ if there exists a collared proper topological embedding of an annulus into a topological manifold homeomorphic to $S^3\times [0,1]$ which restricts on its boundary to the given knots. $\C^{ex}$, short for $\C^{exotic}$, is the set of equivalence classes of knots where two are equivalent if they cobound a properly, smoothly embedded annulus in a smooth manifold \textit{homeomorphic} to $S^3\times [0,1]$; that is, they are concordant in $S^3\times [0,1]$ equipped with a possibly exotic smooth structure. This has been called \textit{pseudo-concordance} by some authors ~\cite{Boyer1}~\cite[Def. 2]{SatoY}. If the smooth $4$-dimensional Poincar\'{e} Conjecture is true then $\C^{ex}=\C$. Finally, for a fixed non-zero integer $n$,  $\C^{\frac{1}{n}}$ denotes the set of equivalence classes of knots in $S^3$ where two  are equivalent if they cobound a smoothly embedded annulus in a smooth $4$-manifold that is a $\Z[\frac{1}{n}]$-homology $S^3\times [0,1]$. For odd $n$ it seems to be unknown whether $\C=\C^{\frac{1}{n}}$! The latter could also be considered in the topological category and our results hold, but we suppress this.  For economy we will  use the notation $\C^{*}$ to denote $*=top$, $*=ex$ or $*=\frac{1}{n}$, reserving the notation $\C$ for the smooth knot concordance group. It is easy to see that each of these is an abelian group under connected sum. In each case the identity is the class of the trivial knot $U$, and the inverse of $K$, denoted $-K$, is the reverse of the mirror image of $K$, denoted $r\overline{K}$.  If $K=0=U$ in $\C$ (respectively: $\C^{ex}, \C^{top}, \C^{\frac{1}{n}})$ then $K$ is called a (smooth) \textbf{slice knot} (respectively: pseudo-slice, topologically slice, $\Z[\frac{1}{n}]$-slice). This is equivalent to saying that $K$ bounds a smoothly embedded disk in a manifold diffeomorphic to $B^4$ (respectively: bounds a smoothly embedded disk in an exotic $B^4$, bounds a collared, topologically embedded disk in a manifold  homeomorphic to $B^4$, bounds a smoothly embedded disk in a smooth manifold that is $\Z[\frac{1}{n}]$-homology equivalent to $B^4$).

We are interested in whether such satellite operators are injective functions (beware they are not homomorphisms). Call such an operator \textbf{weakly injective} if $P(K)=P(0)$ implies $K=0$ (here $0$ is the class of the trivial knot). It is a long-standing open problem whether the Whitehead double operator is weakly injective on $\C$ ~\cite[Problem 1.38]{Kirbyproblemlist}. Considerable effort has been expended in providing evidence for this conjecture  (see ~\cite{HedKirk2} for a survey and the most recent results).  There has recently been speculation that many other ``non-trivial'' satellite operators are injective on $\C$. In ~\cite{CHL5} large classes of winding number zero operators called ``robust doubling operators'' were introduced and evidence was presented for their injectivity.  Yet no single ``non-trivial'' operator is known to be even weakly injective (the exception being the degenerate ``connected-sum operator''  which arises when $P$ intersects the meridional disk of $ST$ in a single point).

Here we have more success for non-zero winding number operators, especially winding-number $\pm 1$ operators. In fact in this paper we will need a stronger version of winding number $\pm 1$. By viewing $ST$ as the standard unknotted solid torus in $S^3$, we arrive at another knot in $S^3$ via $P\hookrightarrow ST\hookrightarrow S^3$. This knot will be denoted by $\widetilde{P}$. Note that $\widetilde{P}=P(U)$ where $U$ is the trivial knot.  For example, for the $P$ shown in Figure~\ref{fig:satelliteexample}, $\widetilde{P}$ is the trivial knot.  The \textbf{winding number} of $P$ is the algebraic intersection number of $P$ with a meridional disk of $ST$.   Let $\eta$ denote the oriented meridian of $ST$, $\{1\}\times \partial D^2$.  The condition that a pattern $P$ has winding number $\pm 1$ is equivalent to the condition that $\eta$  generates $H_1(S^3-\widetilde{P})$.

\begin{defn}\label{def:strongwinding} The pattern $P$ has \textbf{strong winding number} $\pm 1$ if the meridian of the solid torus $ST$ normally generates $\pi_1(S^3-\widetilde{P})$.
\end{defn}

The example in Figure~\ref{fig:satelliteexample} has strong winding number one. Our main theorem is:

\newtheorem*{thm:mainresult}{Theorem~\ref{thm:mainresult}}
\begin{thm:mainresult} Suppose $P$ is a pattern with non-zero winding number $n$. Then
\begin{itemize}
\item [a.]  $P:\C^{\frac{1}{n}}\to \C^{\frac{1}{n}}$ is an injective function.\\
Suppose that $P$ is a pattern with strong winding number $\pm 1$. Then 
\item [b.]  $P:\C^{ex}\to \C^{ex}$ is an injective function,
\item [c.]   $P:\C^{top}\to\C^{top}$ is an injective function, and
\item [d.]  if $S^4$ has a unique smooth structure (up to diffeomorphism) then $P:\C\to\C$ is an injective function.
\end{itemize}
\end{thm:mainresult}

This establishes that the sets $\C^*$  admit many natural self-similarities (as conjectured in ~\cite{CHL5}) ~\cite[Def. 3.1]{BGN}. 

Restricting part  $a.$ of the theorem to the case $n=1$ yields the following simple result:

\newtheorem*{cor:mainresult1}{Corollary~\ref{cor:mainresult1}}
\begin{cor:mainresult1} Suppose $P$ is a pattern with winding number $\pm 1$. Then $P(K)$ is smoothly concordant to $P(J)$ in a smooth homology $S^3\times [0,1]$ if and only if $K\#-J$ is smoothly slice in a smooth homology $B^4$.
\end{cor:mainresult1}

Similarly, restricting part $a.$ to cable operations yields:

\newtheorem*{cor:mainresult2}{Corollary~\ref{cor:mainresult2}}
\begin{cor:mainresult2}  If $p$ and $q$ are coprime positive integers then the $(p,q)$ cable of $K$ is  smoothly concordant to the $(p,q)$ cable of $J$ in a smooth $\Z[\frac{1}{p}]$-homology $S^3\times [0,1]$ if and only if $K$ is smoothly concordant to $J$ in a smooth $\Z[\frac{1}{p}]$-homology $S^3\times [0,1]$.
\end{cor:mainresult2}

The case $p=2$ of Corollary~\ref{cor:mainresult2}  was proved previously by the third author and indeed was one of the inspirations for the current paper ~\cite{Aru1}. The current paper also owes a substantial debt to the techniques of ~\cite{CFHH}. Our techniques are elementary. We use only basic topology and handlebody techniques, except for our use of Freedman's proof of the $4$-dimensional topological Poincar\'{e} Conjecture.

In Section~\ref{sec:stringlinks} we extend some of our results to links. In Section~\ref{sec:questions} we pose a few questions.

\section{Satellite knots and strong winding number one patterns}\label{sec:satellitesandwinding}

In this section we review the formal definition of satellite operators and collect a few elementary properties of satellite operators that we will need in the proof of the main theorem. We also investigate the concept of strong winding number $\pm 1$ and indicate how many operators have this property.

Let $ST\equiv S^1\times D^2$ where both $S^1$ and $D^2$ have their usual orientations. We will always think of $ST$ as embedded in $S^3$ in the standard unknotted fashion. Suppose $P\subset ST$ is an embedded oriented circle that is geometrically essential (even after isotopy $P$ has non-trivial intersection with a meridional $2$-disk). Suppose $K$ is an oriented knot in $S^3$ given as the image of the embedding $f_K:S^1\to S^3$. Then there is an orientation-preserving  diffeomorphism $\tilde{f}_K:S^1\times D^2\to N(K)$, where $N(K)$ is a tubular neighborhood of $K$, such that $\tilde{f}_K=f$ on $S^1\times\{0\}$ and $\tilde{f}_K$ takes the oriented meridian $\eta$ of $ST$ to the oriented meridian of $K$, and takes a preferred  longitude of $ST$, $S^1\times \{1\}$, to a preferred oriented longitude of $K$. The (oriented) knot type of the image of $P$ under $\tilde{f}_K: ST \to N(K)\hookrightarrow S^3$ is called the (untwisted) \textbf{satellite} of $K$ with \textbf{pattern knot} $P$  ~\cite[p. 10]{Lick2}. This will be denoted $P(K)$. 
\begin{figure}[htbp]
\setlength{\unitlength}{1pt}
\begin{picture}(252,200)
\put(50,0){\includegraphics{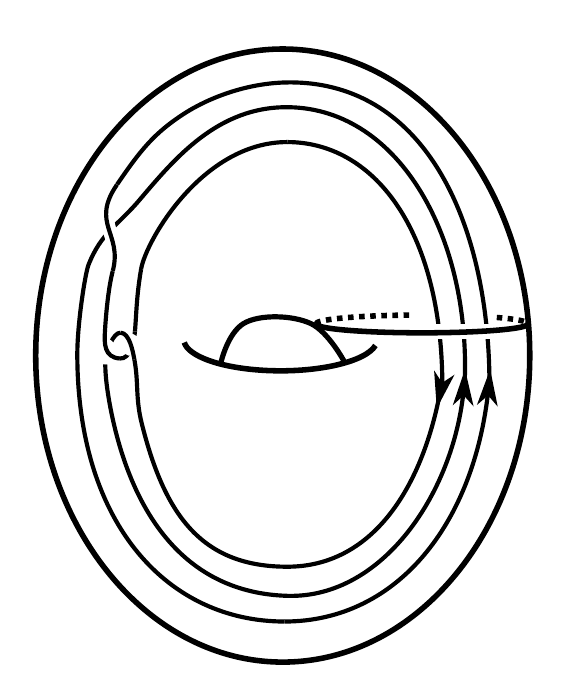}}
\put(214,108){$\eta$}
\put(120,45){$P$}
\end{picture}
\caption{A strong winding number one pattern $P$}
\label{fig:satelliteexample}
\end{figure}
In this paper $P$ will denote, depending on the context, either a knot in the solid torus or the corresponding induced function on a set of equivalence classes of knots in $S^3$:
$$
P:\K/\sim\to \K/\sim
$$
given by $K\mapsto P(K)$.  Such functions seem rarely to be additive with respect to the monoidal structure on $\K$ given by connected sum.  It is well known that satellite functions descend to  yield what we call \textbf{satellite operators}, on $\mathcal{K}/\sim$ for various important equivalence relations on knots, In particular any such operator descends to  $P:\mathcal{C}^*\to \mathcal{C}^*$ on the various sets of ``concordance classes of knots'' as defined in Section~\ref{sec:intro}. For a fixed pattern knot $P$, we will use the same notation for each of these satellite operators.

We present a few elementary results to indicate that strong winding number $\pm 1$ patterns are plentiful. 

\begin{prop}\label{prop:unknottedwindingone} If $P$ is a pattern with $\widetilde{P}$ unknotted then strong winding number $\pm 1$ is equivalent to winding number $\pm 1$.
\end{prop}
\begin{proof} If  $\widetilde{P}$ is  unknotted, $\pi_1(S^3-\widetilde{P})\cong H_1(S^3-\widetilde{P})$. Thus if $P$ has winding number $\pm 1$ then $\eta$ generates $\pi_1(S^3-\widetilde{P})$, so $P$ has strong winding number $\pm 1$.
\end{proof}

\begin{cor}\label{cor:linkstrong} Any two component link with linking number $\pm 1$ and each component unknotted corresponds to a pattern with strong winding number $\pm 1$.
\end{cor}
\begin{proof} If $(\widetilde{P},\eta)$ is such a link then $S^3-N(\eta)$ is a solid torus containing $\widetilde{P}$ and hence defines a pattern of winding number $\pm 1$. Then apply Proposition~\ref{prop:unknottedwindingone}.
\end{proof}

Many other examples of strong winding number one patterns have appeared in the literature in the context of the study of Mazur manifolds (see for example ~\cite[Figure 2]{AkKi3}).

But what if $J=\widetilde{P}$ is knotted? Are there many examples of pairs $(J,\eta)$ where $\eta$ is a normal generator of $\pi_1(S^3-J)$? In his thesis and later in ~\cite{Tsau}, C. Tsau studied this question. He called a  class $\eta$ that normally generates $\pi_1(S^3-J)$ a ``knot killer''.   Note that any such (free) homotopy class has (many) embedded representatives that are unknotted in $S^3$, leading to (presumably distinct) strong winding $\pm 1$ operators.  One way to get knot killers is to take the image of a meridian under an automorphism of $\pi_1$. But if $J$ is prime and not a cable knot or a torus knot then such an automorphism must be induced by a  homeomorphism and hence must send the meridian to a conjugate of itself or its inverse. Thus these so-called ``algebraic knot killers'' are rather restrictive. Tsau proved the existence of a ``non-algebraic knot killer''. This last concept was renamed: \textit{pseudo-meridian},  in ~\cite{SilWhitWill}  where it was shown that the group of any non-trivial torus knot,  $2$-bridge knot or hyperbolic knot with unknotting number one contains infinitely many pseudo-meridians (none equivalent to the other under an automorphism of the group or by conjugation). Thus there exists a large number of strong winding number $\pm 1$ operators.

There is a bijection between satellite operators and ordered, oriented $2$-component links $L=(K_0,K_1)$, for which $K_1$ is unknotted. This is obtained by thinking of the solid torus $ST$ as embedded in $S^3$ in the standard fashion and then setting $(K_0,K_1)=(\widetilde{P},\eta)$.  The following result is important to keep in mind. There is an analogous result in each of the categories.

\begin{prop}\label{prop:sameoperators} If $L_0$ and $L_1$ are two such links that are concordant in $S^3\times [0,1]$ (even with an exotic smooth structure) then the corresponding operators $P_{0}$ and $P_{1}$ are identical functions on $\C^{ex}$.
\end{prop}
\begin{proof} Suppose $A_0$ and $A_1$ are disjointly, smoothly embedded annuli in $S^3\times [0,1]$ (possibly with an exotic smooth structure) that exhibit the concordance. Replace a suitable tubular neighborhood of $A_1$ with $(S^3-K)\times [0,1]$. Then the annulus $A_0$ is a smooth concordance between $P_0(K)$ and $P_1(K)$ in 
a smooth $4$-manifold which may be checked to be homeomorphic to $S^3\times [0,1]$. Thus $P_0=P_1$ as functions on $\C^{ex}$.
\end{proof}

Thus if every $2$-component link with linking number one and one component unknotted were concordant to the Hopf link (equal in $\C^*$)  then Proposition~\ref{prop:sameoperators} would show that the resulting operator would be equal to that of the Hopf link. Since  a Hopf link with a local knot tied in the first component corresponds to the ``connected sum''  operator (the identity operator for the Hopf link itself), and since the connected sum operator is clearly injective, our main result for strong winding number one operators would follow. Therefore it is important to note that there are many such links that are not concordant to the Hopf link as evidenced by several recent papers ~\cite{ChaKimRubSt, FrPo}.  Thus there exist a very large number of strong winding number $\pm 1$ operators that are 
(presumably) distinct from the trivial ``connected-sum'' operator. Moreover there are more subtle invariants from Heegard Floer homology and Khovanov homology that have been used in the smooth category to show that even the simplest winding number one operators, $P$, are not equivalent to the connect-sum operator since there are knots for which $P(K)$ and $K$ have different $\tau$ invariants ~\cite[Section 3]{CFHH}. This is also (indirectly) the main point of  
papers such as ~\cite{AkYa1}, namely that even the simplest winding number one satellites change basic invariants and hence are \textit{not} equivalent to a trivial operator or a connected-sum operator. Therefore  Theorem~\ref{thm:mainresult} has significant content.

Henceforth we will adopt a more schematic representation of a general satellite knot as shown in Figure~\ref{fig:schematicpattern}. 
\begin{figure}[htbp]
\setlength{\unitlength}{1pt}
\begin{picture}(152,120)
\put(15,0){\includegraphics{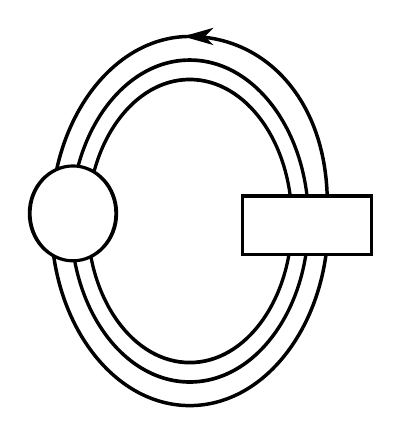}}
\put(32,61){$P$}
\put(98,57){$K$}
\end{picture}
\caption{P(K)}
\label{fig:schematicpattern}
\end{figure}
Such a picture will be used to represent an \textit{arbitrary} satellite knot $P(K)$ for a pattern $P$ of \textit{arbitrary} winding number. The $P$ inside a round disk should be thought of as an arbitrary tangle which closes up to be connected. The knot $K$ inside a rectangle may be understood to symbolize the following string link. Suppose $m$ strands pass through the rectangle. Start with an arc in $B^2\times [0,1]$ knotted in the shape of $K$ (a $1$-component string link). Then form $m$ parallel untwisted copies of this knotted arc. This is the string link to be inserted in the rectangle. This is the same as ``tying all the strands into the knot $K$'',  which in turn yields the knot type of $P(K)$.

The following identity is obvious, once the right hand side is properly interpreted. It will be needed in the proof of our main theorem.

\begin{lem}\label{lem:A}  The identity $P(A\#B)=P(A)(B)$ holds in $\mathcal{K}$ and hence in any $\mathcal{K}/\sim$.
\end{lem}
\begin{proof} The proof is in Figure~\ref{fig:Lemmaconnectsum}.  Given a pattern $P$ and knots $A$ and $B$, the knot $P(A)$ may also be considered to be a pattern, and hence an operator, as in Figure~\ref{fig:Lemmaconnectsum} c). By abuse of notation we use $P(A)$ to also denote this operator. This operator may then act on a knot $B$ yielding $P(A)(B)$, as shown in Figure~\ref{fig:Lemmaconnectsum} b). But this is easily seen to be isotopic to the knot $P(A\#B)$ as shown in Figure~\ref{fig:Lemmaconnectsum} a).
\end{proof}
\begin{figure}[htbp]
\setlength{\unitlength}{1pt}
\begin{picture}(302,180)
\put(-65,15){\includegraphics{mirror1}}
\put(75,15){\includegraphics{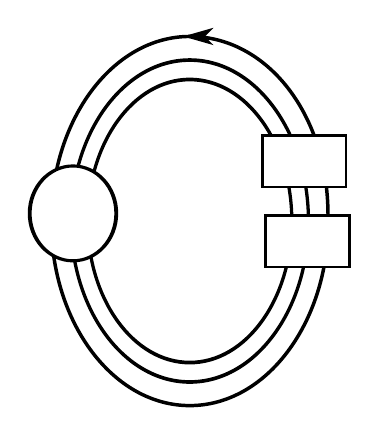}}
\put(225,15){\includegraphics{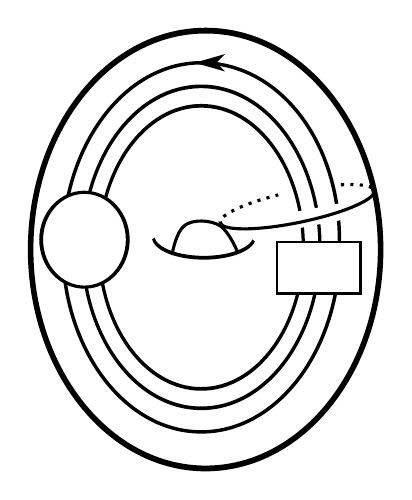}}
\put(-48,76){$P$}
\put(12,72){$A\#B$}
\put(91,76){$P$}
\put(158,90){$B$}
\put(158,67){$A$}
\put(244,86){$P$}
\put(312,78){$A$}
\put(-40,0){a) $P(A\#B)$}
\put(100,0){b) $P(A)(B)$}
\put(240,0){c) the operator $P(A)$}
\end{picture}
\caption{}
\label{fig:Lemmaconnectsum}
\end{figure}

We will also need the following in the proof of our main theorem. In this proposition $P(K)$ is meant as a pattern as in Figure~\ref{fig:Lemmaconnectsum}
c).
\begin{prop}\label{prop:satstrongwinding} If the pattern $P$ has strong winding number $\pm1$ then, for any $K$, so does the pattern $P(K)$.
\end{prop}
\begin{proof}  Let $\eta, \lambda$ denote a meridian and longitude, respectively,  of a solid torus $ST$ associated to the pattern $P$. The kernel of the epimorphism
$$
\pi_1(ST-P)\to \pi_1(S^3-\widetilde{P})
$$
is normally generated by $\lambda$. By hypothesis, $\pi_1(S^3-\widetilde{P})$ is normally generated by the image of $\eta$. It follows that $\pi_1(ST-P)$ is normally generated by $\{\eta,\lambda\}$.

View the exterior of the satellite knot $P(K)$ as the union of  $S^3-K$ and $ST-P$, identified along $\partial(ST)$. Thus $\pi_1(S^3-P(K))$ is normally generated by $\{\eta,\lambda,\mu_{K}\}$. But $\lambda$, lying on the boundary of $S^3-K$, is in the normal closure of $\mu_K$ so it is redundant. Finally, $\eta$ is identified with $\mu_K$ so the latter is redundant. Hence 
$\pi_1(S^3-P(K))$ is normally  generated by $\eta$.

Now let $\eta'$ be a meridian of  the solid torus $ST'$  in which is contained the pattern $P(K)$. This solid torus is bounded by the darker torus in Figure~\ref{fig:strongwinding}. The solid torus $ST$ is bounded by the smaller torus shown dashed in Figure~\ref{fig:strongwinding}. The meridian $\eta'$ is isotopic to the meridian $\eta$ in $S^3-P(K)$. Hence $\pi_1(S^3-P(K))$ is normally  generated by $\eta'$. Thus the operator $P(K)$ has strong winding number $\pm1$.
\begin{figure}[htbp]
\setlength{\unitlength}{1pt}
\begin{picture}(182,130)
\put(30,-10){\includegraphics{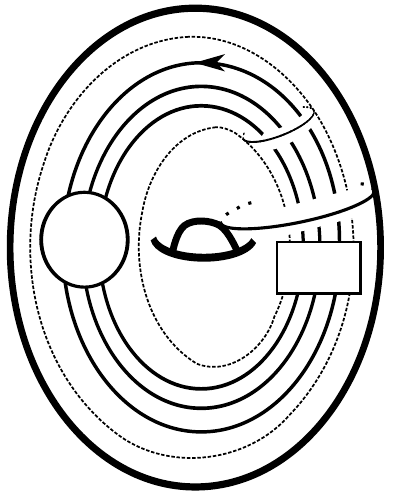}}
\put(143,75){$\eta'$}
\put(118,54){$K$}
\put(95,87){$\eta$}
\put(49,62){$P$}
\end{picture}
\caption{}
\label{fig:strongwinding}
\end{figure}
\end{proof}

\section{Knot concordance and homology cobordism}\label{sec:homologycobordism}

An important ingredient in our proof is a well-known relationship between concordance of knots and homology cobordism of certain $3$-manifolds associated to the knots via surgery. Specifically, given a knot $K$ in $S^3$ we may associate to it the closed oriented $3$-manifold, $M(K)$, called the \textbf{zero-framed surgery} on $S^3$ along  $K$. This is obtained by removing from $S^3$ a tubular neighborhood of $K$ and then replacing it differently, in such a way that the longitude of $K$ bounds the meridional disk of the solid torus.  Let $\mathcal{M}$ denote the set of oriented diffeomorphism classes of $3$-manifolds. Then zero-framed surgery may be viewed as a function $\widetilde{M}$ as shown in Diagram~\ref{diagram:zerosurgery}. If $K$ is concordant to $J$ in $S^3\times [0,1]$ then one can do zero-framed surgery along the connecting annulus and see that $M(K)$ is $\Z$-homology cobordant to $M(J)$ via a $4$-manifold $V$ whose $\pi_1$ is normally generated by $\pi_1$ of either of its boundary components. Therefore $\widetilde{M}$ descends to a well-defined zero-surgery function,  $M$, as shown in Diagram~\ref{diagram:zerosurgery}, where $\mathcal{HC}^*=\mathcal{M}/\sim$ is defined as follows.
\begin{equation}\label{diagram:zerosurgery}
\begin{diagram}\dgARROWLENGTH=2.0em
\node{\mathcal{K}} \arrow{e,t}{\widetilde{M}}\arrow{s,t}{\pi}\node{\mathcal{M}}\arrow{s,t}{\pi}\\
\node{\mathcal{C}^*} \arrow{e,t}{M}\node{\mathcal{HC}^*}
\end{diagram}
\end{equation}
\begin{defn}\label{def:hcrelations}  Suppose $X,Y\in \mathcal{M}$. We say $X\sim Y$ in $\mathcal{HC}^{ex}$ (respectively, $\mathcal{HC}^{top}$)  if $X$ and $Y$ are smoothly (respectively, topologically) homology cobordant via a $4$-manifold $V$ for which $\pi_1(V)$ is normally generated by $\pi_1$ of either boundary component. We say $X\sim Y$ in $\mathcal{HC}^{\frac{1}{n}}$ if $X$ and $Y$ are smoothly $\Z[\frac{1}{n}]$-homology cobordant.
\end{defn}

It is interesting to ask to what extent such surgery functions are injective or weakly injective (note however that a knot and its reverse necessarily have the same image). For example, the weak injectivity of $\widetilde{M}$ is the famous Property $R$ for knots, proved by Gabai ~\cite{Ga87}. In this paper we require the analogue of Property $R$ for concordance. This is a much easier and well-known result. It says that $K$ being zero in $\C^*$ admits a characterization in terms of the zero-framed surgery $M(K)$ (see ~\cite{KM3} for a similar but weaker result).

\begin{prop}\label{prop:homcobordism}  The function $M$ is weakly injective in all $3$ categories $*$, that is, $M(K)\sim M(0)$ if and only if $K=0$. If the smooth $4$-dimensional Poincar\'{e} Conjecture is true then $M:\C\to\C$ is weakly injective.
\end{prop}
 
\begin{proof}  This result is discussed in more detail in ~\cite[Prop. 1.2, Prop. 1.5]{CFHH}. Since this result and its proof are by now well-known, we merely sketch, for the convenience of the reader, the proof in the category $*=ex$. Suppose $M(K)\sim M(0)$, that is $M(K)$ is smoothly homology cobordant to $S^1\times S^2$, with the extra $\pi_1$ condition. Then, after capping off with $S^1\times B^3$, we see that $M(K)=\partial W$ where the pair $(W, M(K))$ is a smooth $\Z$-homology $(S^1\times B^3,S^1\times S^2)$ with the additional property that $\pi_1(W)$ is normally generated by the meridian of $K$. Let $\mathcal{B}$ be the $4$-manifold obtained from $W$ by adding a $2$-handle along the meridian of $K$ in $M(K)$. Then $\mathcal{B}$ is a smooth, contractible $4$-manifold whose boundary is $S^3$. By work of M. Freedman, $\mathcal{B}$ is homeomorphic to $B^4$ ~\cite{Freed1}. Moreover the co-core of the $2$-handle is a smooth slice disk for $K$. Thus $K$ is smoothly concordant to the trivial knot in a smooth manifold that is homeomorphic to $S^3\times [0,1]$. This shows the difficult direction of the first sentence of Proposition~\ref{prop:homcobordism} in the case $*=ex$.

For the last sentence of Proposition~\ref{prop:homcobordism}, we cap off $\mathcal{B}$ with $B^4$ and arrive at $\Sigma$, a smooth homotopy $4$-sphere, which under our hypothesis is diffeomorphic to $S^4$. Hence there is a smoothly embedded $3$-sphere in $S^4$ whose complementary components are diffeomorphic to $\mathcal{B}$ and $B^4$ respectively. But this special case of the smooth Schoenflies problem is known: if one complementary component is diffeomorphic to $B^4$ then the other is also. The sketch of the proof is that the smooth embedding $B^4\hookrightarrow S^4$ is isotopic to the standard embedding as the upper-hemisphere ~\cite[Theorem 3.34]{RourkeSanderson}. Thus, by the isotopy extension theorem, the diffeomorphism $\Sigma\cong S^4$ is isotopic to one sending $\mathcal{B}$ diffeomorphically to the lower hemisphere. Hence $\mathcal{B}$ is diffeomorphic to $B^4$ so $K$ is smoothly slice as desired.
\end{proof}

\section{Satellite knots and homology cobordism}\label{sec:satellitesandhomologycobordism}

A crucial ingredient in the proof of Theorem~\ref{thm:mainresult} is a strengthening of a recent result of Cochran-Franklin-Hedden-Horn (concerning the \textit{failure} of injectivity of $M:\C^*\to \mathcal{HC}^*$ on certain satellites of non-zero winding number!)   Their result is given below.  We note that it is the failure of the results in this section for patterns of winding number zero that prevents us from proving  injectivity for winding number zero operators. 

\begin{thm}~\cite[Thm. 2.1]{CFHH}\label{thm:CFHHmainsatelliteresult}  Suppose $P$ is a pattern with non-zero winding number $n$ such that $\widetilde{P}\equiv P(U)$ is smoothly slice in a $\Z[\frac{1}{n}]$-homology ball. Then, for any knot $K$,  $M(P(K))$ is smoothly $\Z\left[\frac{1}{n}\right]$-homology cobordant to $M(K)$. Hence $M(P(K))\sim M(K)$ in $\mathcal{HC}^{\frac{1}{n}}$.
\end{thm}

A special case of Theorem~\ref{thm:mainresult}  a) (weak injectivity for slice operators) follows:

\begin{cor}\label{cor:CFHH} Suppose $P$ is a pattern with non-zero winding number $n$ such that $\widetilde{P}$ is smoothly slice in a $\Z[\frac{1}{n}]$-homology ball. Then
$$
P:\C^{\frac{1}{n}}\to\C^{\frac{1}{n}}
$$
is weakly injective.
\end{cor}
\begin{proof}[Proof of Corollary~\ref{cor:CFHH} ] Suppose $P(K)= P(0)$ in $\C^{\frac{1}{n}}$. Since $P(0)=\widetilde{P}=0$, $P(K)=0$. Thus $M(P(K))\sim M(0)$ in $\mathcal{HC}^{\frac{1}{n}}$.  By Theorem~\ref{thm:CFHHmainsatelliteresult} , we also have $M(P(K))\sim M(K)$. Hence $M(K)\sim M(0)$. Since the function $M$ is weakly injective by  Proposition~\ref{prop:homcobordism}, it follows that $K=0$ in $\C^{\frac{1}{n}}$.
\end{proof}

In this paper, in the case that $n=\pm 1$,  we will strengthen the hypotheses of Theorem~\ref{thm:CFHHmainsatelliteresult}  in order to get the stronger conclusion that $M(P(K))\sim M(K)$ in $\mathcal{HC}^{ex}$ or $\mathcal{HC}^{top}$. That is, we want to be able to conclude that $M(P(K))$ and $M(K)$ are smoothly homology cobordant via a $4$-manifold $V$ whose $\pi_1$ is normally generated by that of either of its boundary components. 

\begin{thm}\label{thm:CFHHmainstronger}  Suppose $P$ is a pattern with winding number $\pm 1$ such that $\widetilde{P}$ is a pseudo-slice knot (respectively, topologically slice knot) and such that the meridian of $ST$ normally generates $\pi_1(B^4-\Delta)$ where $\Delta$ is a slice disk for $\widetilde{P}$. Then, for any knot $K$,  $M(P(K))\sim M(K)$ in $\mathcal{HC}^{ex}$ (respectively in $\mathcal{HC}^{top}$). 
\end{thm}

Since for a strong winding number $\pm 1$ pattern, $\mu_{\tilde{P}}$ lies in the normal closure of $\eta$ in $\pi_1(S^3-\widetilde{P})$, and since $\mu_{\tilde{P}}$ normally generates $\pi_1(B^4-\Delta)$, Theorem~\ref{thm:CFHHmainstronger}  implies:

\begin{cor}\label{cor:CFHHmainstronger}  Suppose $P$ is a pattern with strong winding number $\pm 1$ such that $\widetilde{P}$ is a pseudo-slice knot (respectively, topologically slice knot). Then, for any knot $K$,  $M(P(K))\sim M(K)$ in $\mathcal{HC}^{ex}$ (repectively in $\mathcal{HC}^{top}$).
\end{cor}

A special case of Theorem~\ref{thm:mainresult}  $b)$ and $c)$ (weak injectivity for certain operators) follows quickly:

\begin{cor}\label{cor:CFHHmainstronger2} Suppose $P$ is a pattern with strong winding number $\pm 1$ such that $\widetilde{P}$ is a pseudo-slice knot (respectively,  a topologically slice knot). Then
$$
P:\C^*\to\C^*
$$
is weakly injective for $*=ex$ (respectively for $*=top$).
\end{cor}

\begin{proof}[Proof of Corollary~\ref{cor:CFHHmainstronger2}]  Suppose $P(K)= P(0)$ in $\C^*$. Since $P(0)=\widetilde{P}=0$, $P(K)=0$. Thus $M(P(K))\sim M(0)$ in $\mathcal{HC}^{*}$.  By Corollary~\ref{cor:CFHHmainstronger}, we also have $M(P(K))\sim M(K)$. Hence $M(K)\sim M(0)$ in $\mathcal{HC}^{*}$. Since the function $M$ is weakly injective by  Proposition~\ref{prop:homcobordism}, it follows that $K=0$ in $\C^*$.
\end{proof}

\begin{proof}[Proof of Theorem~\ref{thm:CFHHmainstronger}] The proof consists of following the proof of ~\cite[Thm. 2.1]{CFHH} while keeping track of $\pi_1$ of the homology cobordism. We will construct a homology cobordism $W$ between $M(K)$ and $M(P(K))$.  Begin with $M(K)\times[0,1]$ and add a $4$-dimensional $1$-handle to $M(K)\times \{1\}$ and let the resulting $4$-manifold be called $W_1$. Then $\partial_+ W_1\cong S^1\times S^2\#M(K)$ as depicted  in Figure~\ref{fig:kirby1} a) (this picture may seem to suggest a $2$-handle has been added along a zero-framed unknotted circle, but recall the latter also leads to a connected sum with $S^1\times S^2$ as the new boundary component- and it is only this $3$-manifold that is being portrayed).  Also shown  (dotted)  are $\mu$ and $\mu_K$, the meridians of the unknot and $K$ respectively. The set $\{\mu,\mu_K\}$ normally generates $\pi_1(W_1)$.
\begin{figure}[htbp]
\setlength{\unitlength}{1pt}
\begin{picture}(300,150)
\put(-45,20){\includegraphics{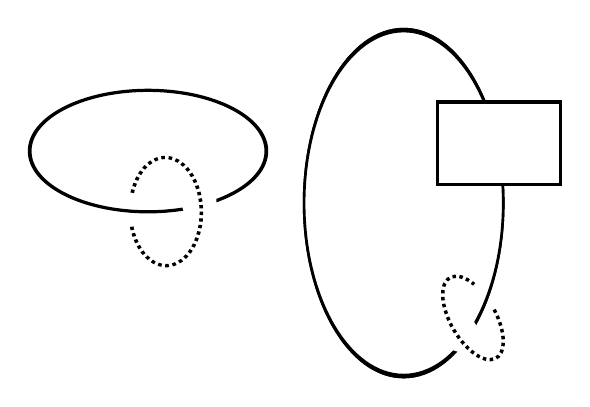}}
\put(145,20){\includegraphics{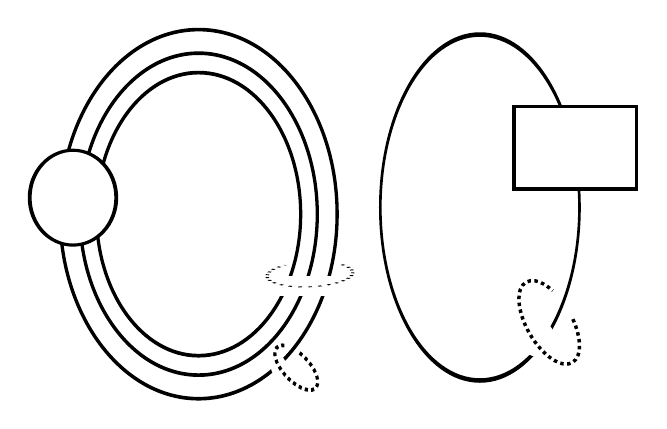}}
\put(93,92){$K$}
\put(5,53){$\mu$}
\put(212,63){$\eta$}
\put(5,120){$0$}
\put(45,120){$0$}
\put(165,120){$0$}
\put(258,125){$0$}
\put(305,98){$K$}
\put(95,25){$\mu_K$}
\put(228,23){$\mu_{\tilde{P}}$}
\put(162,83){$P$}
\put(305,29){$\mu_K$}
\put(-20,0){a) $\partial_+ W_1\cong S^1\times S^2~~\# ~~M(K)$}
\put(180,0){b) $\partial_+ W_2\cong M(\widetilde{P})~~\# ~~M(K)$}
\end{picture}
\caption{}\label{fig:kirby1}
\end{figure}
Since the unknot is smoothly (respectively, topologically) concordant to $\widetilde{P}$ (use the given slice disk $\Delta$), $M(U)$ is smoothly (respectively, topologically) homology cobordant to $M(\widetilde{P})$ via a cobordism whose fundamental group is normally generated by $\mu_{\tilde{P}}$ (alternatively by $\mu $).  Moreover, our additional $\pi_1$ condition implies that $\pi_1$ of this cobordism is normally generated by $\eta$ (see Figure~\ref{fig:kirby1} b) ). Recalling that $M(U)\cong S^1\times S^2$, we conclude that $ S^1\times S^2\# M(K)$ is homology cobordant to $M(\widetilde{P})\# M(K)$ via a cobordism $C$. Let $W_2=W_1\cup C$. Then $\partial_+ W_2=M(\widetilde{P})\# M(K)$ as depicted in Figure~\ref{fig:kirby1} b).    Note that $\pi_1(W_2)$ is normally generated by $\{\mu_K, \mu_{\tilde{P}}\}$, but is also normally generated by $\{\mu_K, \eta\}$.  Now we add a $0$-framed $2$ handle to $\partial_+ W_2$ along the solid zero-framed circle shown in Figure~\ref{fig:kirby2} a),  and call the result $W$. $\partial_+ W$ is depicted in Figure~\ref{fig:kirby2} a).
\begin{figure}[htbp]
\setlength{\unitlength}{1pt}
\begin{picture}(300,150)
\put(-45,20){\includegraphics{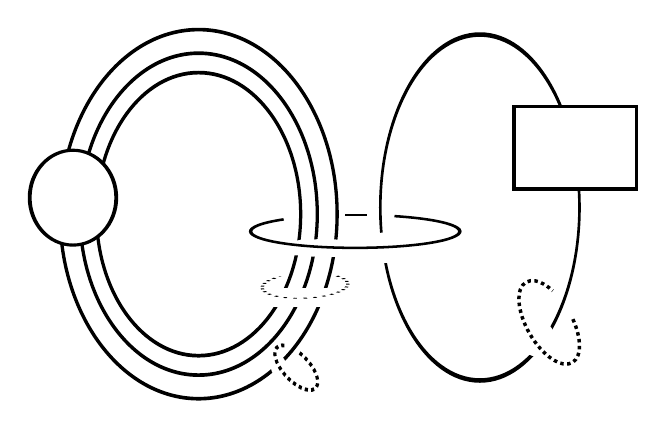}}
\put(175,20){\includegraphics{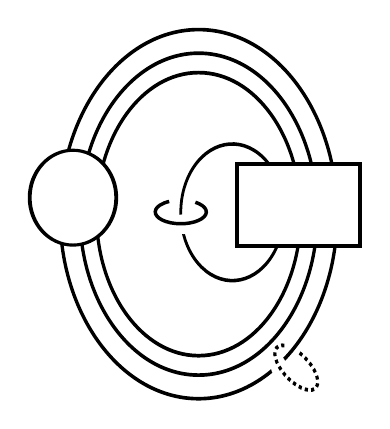}}
\put(114,97){$K$}
\put(24,54){$\eta$}
\put(115,25){$\mu_K$}
\put(-30,120){$0$}
\put(18,72){$0$}
\put(119,120){$0$}
\put(263,125){$0$}
\put(214,72){$0$}
\put(228,99){$0$}
\put(255,22){$\mu_{\tilde{P}}$}
\put(39,21){$\mu_{\tilde{P}}$}
\put(-28,83){$P$}
\put(191,83){$P$}
\put(256,81){$K$}
\put(30,0){a) $\partial_+W$}
\put(203,0){b) $\partial_+ W$}
\end{picture}
\caption{}\label{fig:kirby2}
\end{figure}
We claim that $W$ is the desired homology cobordism. Note that the $2$-handle is added along a loop isotopic to $\eta^{-1}\mu_K$. Since $P$ has winding number $\pm1$, $\eta$ is homologous to $\pm \mu_{\tilde{P}}$ in $\partial_+ W_2$.  Therefore the added $2$-handle equates the new $H_1$-generator $\mu_{\tilde{P}}$ with $\mu_K$. Hence  $H_*(W,M_K)=0$. It follows, by duality and the universal coefficient theorem,  that $W$ is a homology cobordism between $M(K)$ and $\partial_+ W$.

Moreover $\pi_1(W)$ is a quotient of $\pi_1(W_2)$ and hence is normally generated by $\{\mu_K, \eta\}$. Since $\mu_K=\eta$ in $\pi_1(W)$, $\pi_1(W)$ is normally generated by $\mu_K$ alone. We also claim that $\pi_1(W)$ is normally generated by $\mu_{\tilde{P}}$. For certainly $\eta$ is in the normal closure of $\mu_{\tilde{P}}$, and hence $\mu_K$ is also. 

Finally we show that $\partial_+ W\cong M(P(K))$ using the calculus of framed links ~\cite[p.264]{R}. Begin with the framed link description of $\partial_+ W$ given in Figure~\ref{fig:kirby2} a). First ``slide'' each strand of $\widetilde{P}$ that passes through $\eta$ over the $2$-handle marked with $K$ ~\cite{K}. From the point of view of $3$-manifolds this is merely a sequence of isotopies. The result is shown in Figure~\ref{fig:kirby2} b). This shows that the $3$-manifolds depicted in Figures~\ref{fig:kirby2} a) and b) are homeomorphic.  Now we show that the two smaller zero-framed circles can be eliminated entirely resulting in the desired $3$-manifold in Figure~\ref{fig:kirby3}. This is an instance of the so-called slam-dunk move on framed links ~\cite[p.501]{CGompf}. In this move the smallest zero-framed circle is eliminated and the framing of the other circle changed to $-\infty$, which means a solid torus is cut out and replaced in an identical fashion. (For those more adept with the calculus of $4$-manifolds, this can also be justified by changing the smallest $2$-handle to a $1$-handle and then cancelling this $1$-handle with the other $2$-handle.) Moreover it is clear from the proof of the slam-dunk move that the homeomorphism from Figure~\ref{fig:kirby2} b) to Figure~\ref{fig:kirby3} is supported in a neighborhood of the two small circles.  Hence the circle labelled $\mu_{P(K)}$ in Figure~\ref{fig:kirby3}  is isotopic to $\mu_{\tilde{P}}$ in Figure~\ref{fig:kirby2} a), so $\mu_{P(K)}$ normally generates $\pi_1(W)$ as required. Thus $M(K)\sim M(P(K))$ in $\mathcal{HC}^{ex}$ (respectively, in $\mathcal{HC}^{top}$ ).
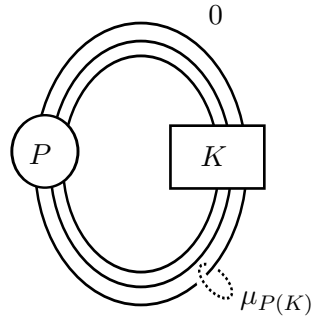
\begin{figure}[htbp]
\setlength{\unitlength}{1pt}
\begin{picture}(300,125)
\put(85,0){\includegraphics{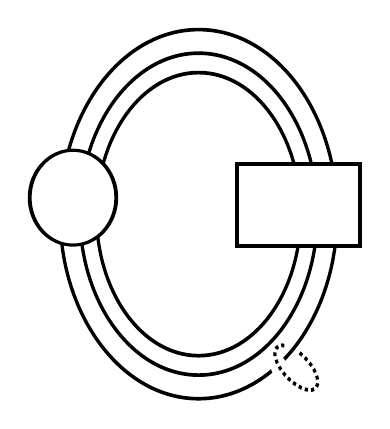}}

\put(165,61){$K$}
\put(100,62){$P$}
\put(180,9){$\mu_{P(K)}$}
\put(168,115){$0$}
\end{picture}
\caption{$\partial_+ W\cong M(P(K))$}\label{fig:kirby3}
\end{figure}

\end{proof}

\section{The proof of Theorem~\ref{thm:mainresult}  }\label{sec:proof}

\begin{thm}\label{thm:mainresult}  Suppose $P$ is a pattern with non-zero winding number $n$. Then
\begin{itemize}
\item [a.]  $P:\C^{\frac{1}{n}}\to \C^{\frac{1}{n}}$ is an injective function.\\
If $P$ is a pattern with strong winding number $\pm 1$ then 
\item [b.]  $P:\C^{ex}\to \C^{ex}$ is an injective function,
\item [c.]   $P:\C^{top}\to\C^{top}$ is an injective function; and
\item [d.]  if $S^4$ has a unique smooth structure (up to diffeomorphism) then $P:\C\to\C$ is an injective function.
\end{itemize}
\end{thm}

\begin{cor}\label{cor:mainresult1}   Suppose $P$ is a pattern with winding number $\pm 1$. Then $P(K)$ is smoothly concordant to $P(J)$ in a smooth homology $S^3\times [0,1]$ if and only if $K\#-J$ is smoothly slice in a smooth homology $B^4$.
\end{cor}

The case $p=2$ of the following corollary was proved previously by the third author ~\cite{Aru1}. 

\begin{cor}\label{cor:mainresult2} If $p$ and $q$ are coprime integers then the $(p,q)$ cable of $K$ is  smoothly concordant to the $(p,q)$ cable of $J$ in a smooth $\Z[\frac{1}{p}]$-homology $S^3\times [0,1]$ if and only if $K$ is smoothly concordant to $J$ in a smooth $\Z[\frac{1}{p}]$-homology $S^3\times [0,1]$.
\end{cor}

\begin{proof}[Proof of Theorem~\ref{thm:mainresult}]  We present a unified proof for parts a), b), and c) of the theorem. We assume that $P$ has non-zero winding number $n$ and in the cases that $*=ex$ and $*=top$ we assume additionally that $P$ has strong winding number $\pm 1$.  Then, under the assumption that $P(K)=P(J)$ in $\C^*$, we will show that $K=J$ in $\C^*$. In fact we will quickly reduce the proof to the special case that $\widetilde{P}=0$ and $J=0$, whose veracity was already established by Corollaries~\ref{cor:CFHH} and ~\ref{cor:CFHHmainstronger2}.

Since $K\#-K$ is smoothly slice, $K\#-K=0$ in $\C^*$ so
$$
P(J)=P(K\#\left(-K\#J\right)).
$$
The right-hand side of this equation can be re-expressed, using Lemma~\ref{lem:A}, as
\begin{equation}\label{eq:proof1}
P(J)=P(K)\left(-K\#J\right).
\end{equation}
By assumption $P(K)=P(J)$ and so, since $\C^*$ is a group, 
$$
-[P(K)]\#P(J)=0.
$$
Substituting for $P(J)$ using Equation~(\ref{eq:proof1}), we have
\begin{equation}\label{eq:proof2}
-[P(K)]\# [P(K)(-K\#J)]=0.
\end{equation}
A picture of the connected-sum of two knots on the left-hand side of Equation~(\ref{eq:proof2}) is shown in Figure~\ref{fig:proof2}. The particular form we have pictured for the $-[P(K)]$ summand is not important. This form will not be used.
\begin{figure}[htbp]
\setlength{\unitlength}{1pt}
\begin{picture}(302,180)
\put(-55,0){\includegraphics{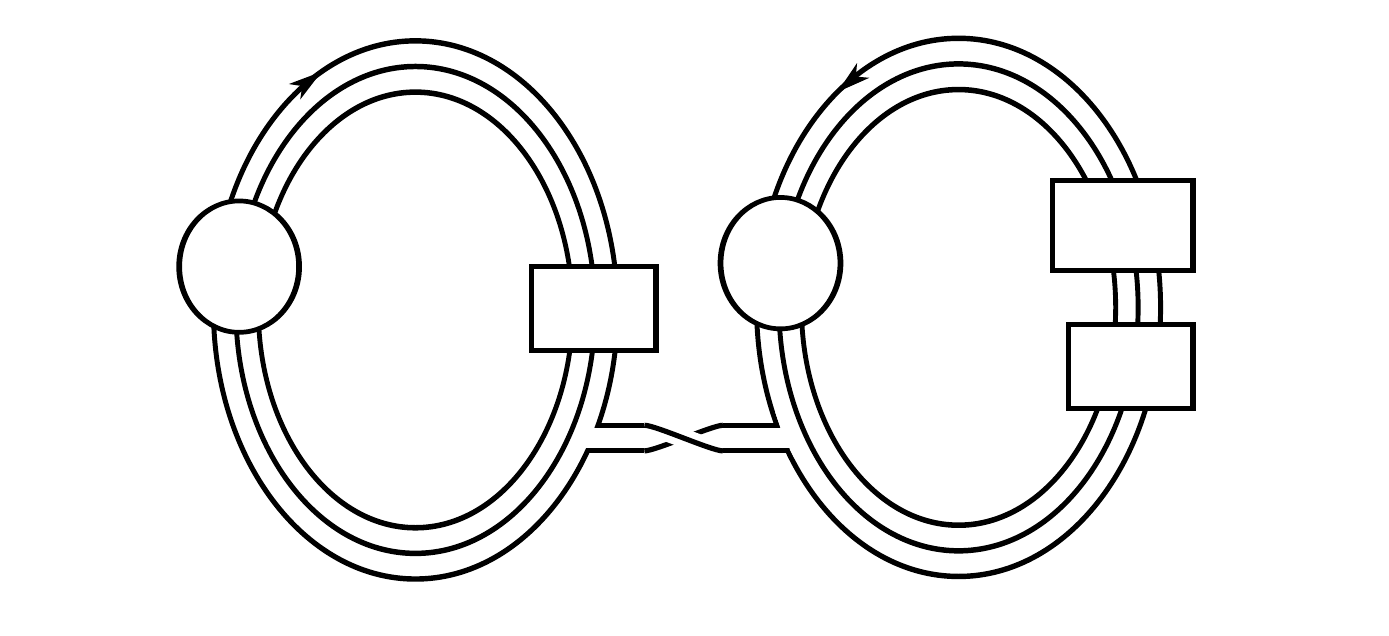}}

\put(8,104){$\overline{P}$}
\put(165,105){$P$}
\put(265,74){$K$}
\put(252,115){$-K\#J$}
\put(112,91){$\overline{K}$}
\end{picture}
\caption{}
\label{fig:proof2}
\end{figure}

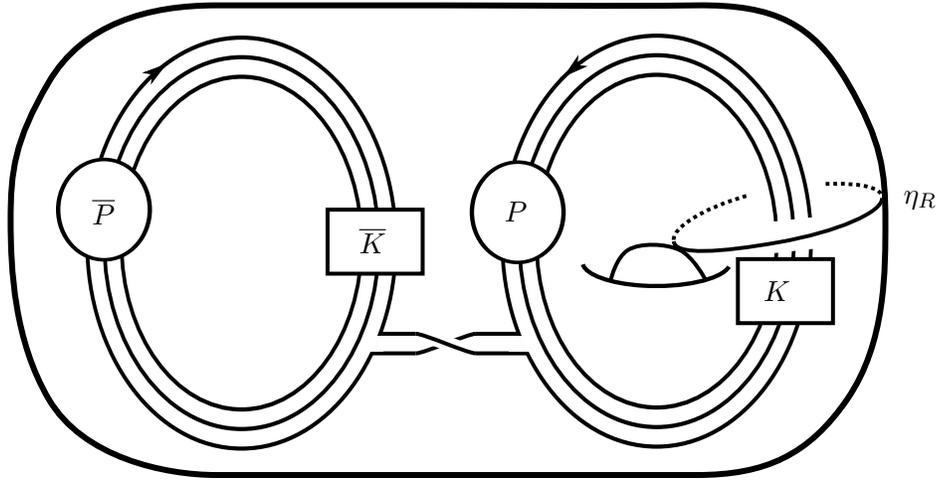
\begin{figure}[htbp]
\setlength{\unitlength}{1pt}
\begin{picture}(302,210)
\put(-27,10){\includegraphics{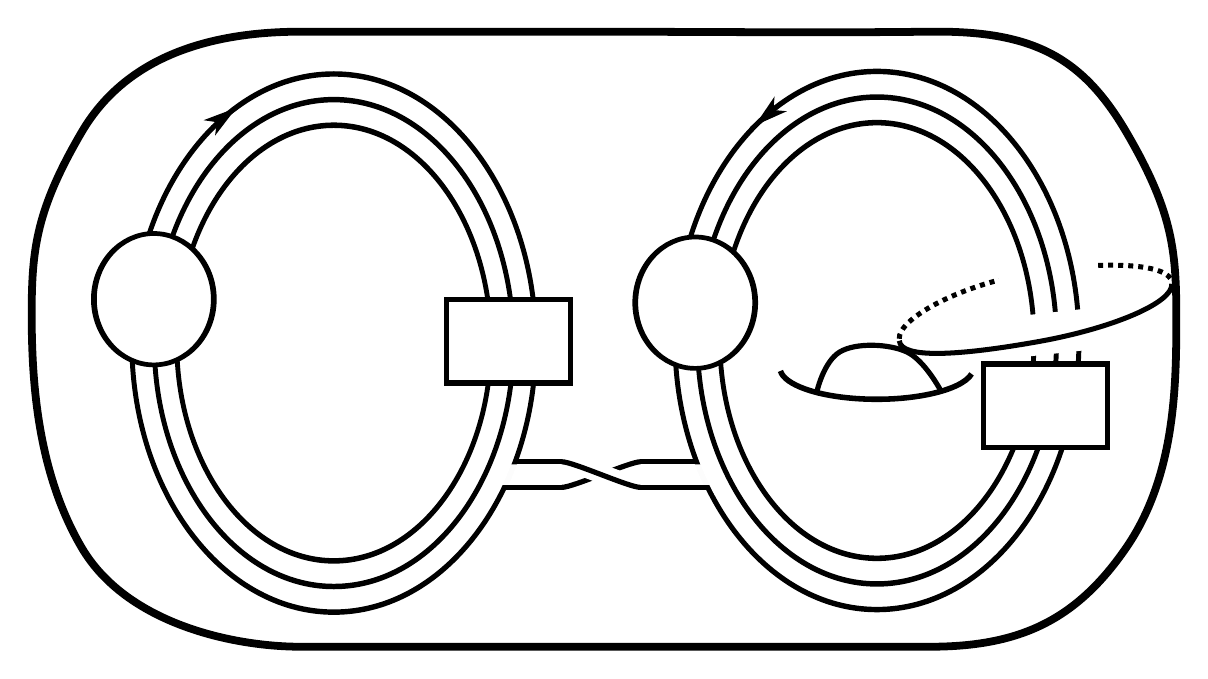}}
\put(13,114){$\overline{P}$}
\put(169,115){$P$}
\put(320,122){$\eta_R$}
\put(267,85){$K$}
\put(114,103){$\overline{K}$}
\end{picture}
\caption{The operator $R=-[P(K)]\# P(K)$}
\label{fig:proof3}
\end{figure}
Let $R$ be the pattern knot shown in Figure~\ref{fig:proof3}. In terms of this operator Equation~(\ref{eq:proof2}) becomes
\begin{equation}\label{eq:proof3}
R(-K\#J)=0.
\end{equation}
Furthermore observe that $\widetilde{R}=-P(K)\#P(K)$ is a ribbon knot hence a smoothly slice knot. Also note that the winding number of $R$ is the same as that of $P$, and hence is non-zero. Now, in the case $*=\frac{1}{n}$, Corollary~\ref{cor:CFHH}  can be applied to the operator $R$ (combined with Equation~\ref{eq:proof3}) to conclude that $-K\#J=0$ in $\C^*$, from which it follows that $K=J$, as required.

In the cases $*=ex$ and $*=top$ we are assuming that $P$ has strong winding number $\pm 1$. We claim that $R$ also has strong winding number $\pm 1$. The verification of this is postponed momentarily. Assuming this, Corollary~\ref{cor:CFHHmainstronger2}  can be applied to the operator $R$ (combined with Equation~\ref{eq:proof3}) to conclude that $-K\#J=0$ in $\C^*$, from which it follows that $K=J$, as required.

 We claim that if $P$  has strong winding number $\pm 1$  then so does $R$. To verify this we must show that the meridian of the knot $\widetilde{R}$, lies in the normal closure of the meridian, $\eta_R$, of the solid torus in Figure~\ref{fig:proof3}. Since the knot group of a connected-sum of knots is a free product of the groups of the factors amalgamated along the meridians, it suffices to show that the meridian of the knot $P(K)$ lies in the normal closure of $\eta_R$. This follows from Proposition~\ref{prop:satstrongwinding}.

This completes the proof of parts a)-c) of Theorem~\ref{thm:mainresult}.  For part d), re-do the proof in the smooth category. The only place where smooth structure becomes an issue is at the very end of the proofs of Corollaries~\ref{cor:CFHH} and ~\ref{cor:CFHHmainstronger2} where we apply Proposition~\ref{prop:homcobordism}. Since the latter  holds in the category $\C$ if the smooth $4$-dimensional Poincar\'{e} Conjecture is true, we are done.
\end{proof}

\section{Extensions to string links and links}\label{sec:stringlinks}

Recall that a (pure) string link $L$ with $m$ components is a smooth proper embedding of a disjoint union of $m$ copies of the oriented unit interval (called strings) into $D^2\times [0,1]$ such that the endpoints of the $i^{th}$ string are sent to a fixed pair of points, say, $((i/2m,0),0)$ and $((i/2m,0),1)$. This is similar to a pure braid without being level-preserving. One can multiply two such string links by the obvious stacking procedure and the identity is the trivial string link with $m$ components. The \textbf{closure} of a string link is the ordered oriented link in $S^3$ obtained by using a trivial string link to identify the top and bottom of the string link. One can also define concordance between string links and so arrive at groups (non-abelian if $m>1$)  of equivalence classes, $\mathcal{C}_m^*$,  just as for knots as discussed in Section~\ref{sec:intro}. In fact $\C_1 \cong \C$. In this section we show how to get many different embeddings
$$
\C^*\hookrightarrow \mathcal{C}_m^*
$$
for $*=\frac{1}{n}$ using generalizations of satellite operators.

Let $\mathcal{L}$, a \textbf{string link operator}, denote a pair $(L,\eta)$ where $L$ is a string link and  $\eta$ is an oriented circle embedded in the exterior of $L$ that is unknotted in $D^2\times [0,1]$. Then given an oriented knot $K$, we can define a new string link, $\cL(K)$, called the \textbf{result of infection on $L$ by $K$ along $\eta$},  as follows. Remove an open tubular neighborhood of $\eta$ from $D^2\times [0,1]$, and replace it by $S^3-\nu(K)$, identifying the meridian of $\nu(\eta)$ with the inverse of a longitude of $K$ and identifying the longitude of $\nu(\eta)$ with the meridian of $K$. After noticing that the resulting manifold is diffeomorphic relative boundary to $D^2\times [0,1]$, we define $\cL(K)$ as the image of $L$ under this diffeomorphism. This correspondence  descends to give a well-defined \textbf{infection operator} or \textbf{string link operator}:
$$
\cL:\C^*\to \mathcal{C}_m^*.
$$
These functions are rarely homomorphisms. We will employ our techniques to find many different examples where such operators are injective. But first we note that there are some obvious such injective functions that  arise, for example,  by taking $\eta$ to be a meridian of the $i^{th}$ component of an arbitrary string link $L$. This operator has the effect of tying a local knot into the $i^{th}$ string and is easily seen to be injective.

Define the \textbf{winding vector} of $\cL$ to be the $m$-tuple $\vec{w}(\cL)=(w_1,...,w_m)$  where $w_i$ is the algebraic number of intersections between the $i^{th}$ string and the disk spanning $\eta$ (the ``linking number'' of the $i^{th}$ string with $\eta$). Define the \textbf{winding number} of $\cL$, denoted $n(\cL)$, to be the greatest common divisor of the coordinates of $\vec{w}(\cL)$ (a positive integer if $\vec{w}\neq \vec{0})$.  

\begin{thm}\label{thm:stringlinkresult}  Suppose $\cL$ is a string link operator with non-zero winding number $n$. Then 
$\cL:\C^{\frac{1}{n}}\to \C_m^{\frac{1}{n}}$ is an injective function. In particular if $n(\cL)=1$ then $\cL(K)$ is concordant to $\cL(J)$ in a homology $D^2\times [0,1]\times [0,1]$ if and only if $K$ is concordant to $J$ in a homology $S^3\times [0,1]$.
\end{thm}

\begin{proof} The first part of the proof is the same as that of Theorem~\ref{thm:mainresult}. Suppose that $\cL(K)=\cL(J)$ in $\C^{\frac{1}{n}}_m$.  Since the latter is a group, we quickly deduce
\begin{equation}\label{eq:linkproof1}
-[\cL(K)]~*~[\cL(K)(-K\#J)]=0,
\end{equation}
where here $*$ means string link multiplication and the minus sign denotes the inverse in $\C_m^{\frac{1}{n}}$.  As in the proof of Theorem~\ref{thm:mainresult}, we can define an infection operator $\R=(R,\eta)$ where
$$
R=-[\cL(K)]~* ~[\cL(K)]
$$
is a smoothly slice string link and 
$$
\R(-K\# J)=0.
$$
Moreover clearly $\vec{w}(\R)=\vec{w}(\cL)$ and $n(\R)=n(\cL)=n$. Therefore we are reduced to showing a string-link analogue of Corollary~\ref{cor:CFHH} :
\begin{prop}\label{prop:stringweak} Suppose $\R=(R,\eta)$ is a string link operator with non-zero winding number  where $R$ is slice (or merely zero in $\C_m^{\frac{1}{n}}$) , then $\R:\C^{\frac{1}{n}}\to \C_m^{\frac{1}{n}}$ is weakly injective.
\end{prop}
\begin{proof}  Suppose $\R(K)=0$. The strategy of the proof is to use parallel copies and fusions to reduce $\R$ to a knot pattern and use Corollary~\ref{cor:CFHH}. 

Since $n(\R)=n$  there is some integral linear  combination $\sum k_iw_i(\R)$ that equals $n$. An example of $\R$ is shown in Figure~\ref{fig:stringlinkfusions} a) where $m=2$, $w_1=2$, $w_2=3$ and $n=1$.  We alter $\R$ as follows.  First we form a new string link by taking $k_i$ parallel untwisted copies of the $i^{th}$ string of $R$ (in the exterior of $\eta$). This means omitting string components for which $k_i=0$ and changing the string orientation of the string if $k_i<0$. Call the resulting string link $R'$ and the resulting operator $\R'=(R',\eta)$. This is carried out for the example in shown in Figure~\ref{fig:stringlinkfusions} b) where we have chosen $k_1=2$ and $k_2=-1$.
\begin{figure}[htbp]
\setlength{\unitlength}{1pt}
\begin{picture}(252,200)
\put(-90,10){\includegraphics{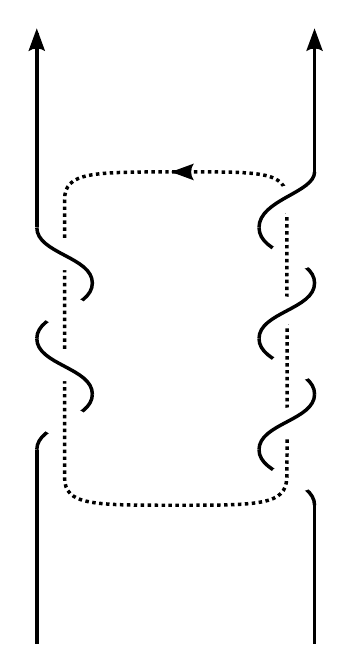}}
\put(70,10){\includegraphics{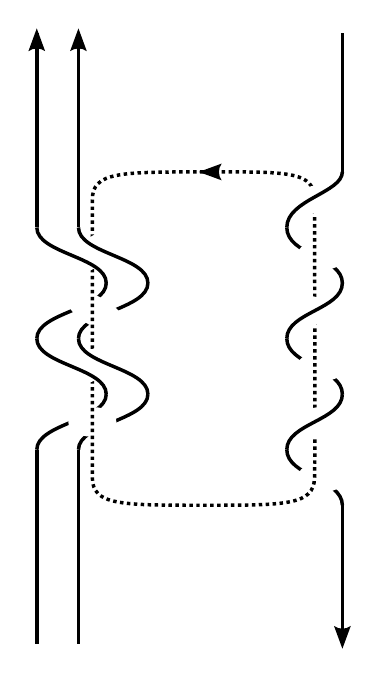}}
\put(240,10){\includegraphics{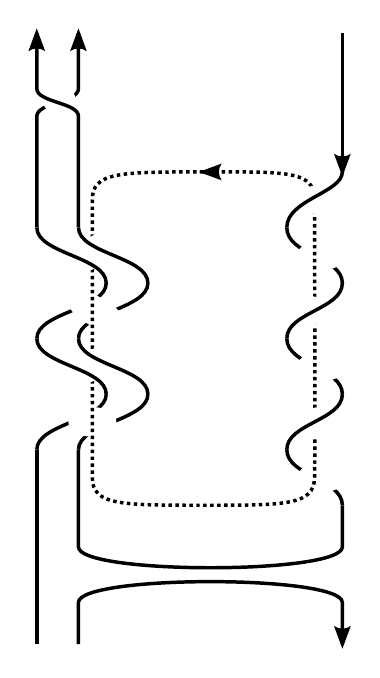}}
\put(-40,48){$\eta$}
\put(126,48){$\eta$}
\put(294,50){$\eta$}
\put(280,0){c) $(T,\eta)$}
\put(95,0){b) $\R'=(R',\eta)$}
\put(-75,0){a) $\R=(R,\eta)$}
\end{picture}
\caption{}
\label{fig:stringlinkfusions}
\end{figure}
This new string link $R'$ will have \textit{total} linking number $n$ with $\eta$. Since by hypothesis $R=0$ in $\C_m^{\frac{1}{n}}$ and since the property of being zero in $\C_m^{\frac{1}{n}}$ is preserved under taking parallels,  we know that $R'=0$.  The two steps $\R\leadsto \R'\leadsto \R'(K)$  ``commute", which means that $\R'(K)$ could also be obtained from $\R(K)$ by taking parallels of components. Since by hypothesis $\R(K)=0$ in $\C_m^{\frac{1}{n}}$ we can also conclude that $\R'(K)=0$. 

Order the components of $R'$. Now we fuse together all the components of $R'$ using bands that do not intersect $\eta$ (fuse component $1$ to component $2$, then component $2$ to component $3$, et cetera), until we arrive at an oriented \textit{tangle}  $(T,\eta)$ whose closure is an oriented \textit{knot} in the exterior of $\eta$. This is shown in Figure~\ref{fig:stringlinkfusions} c) for the example. The closure is not shown. The closure of $T$ represents a (knot) pattern $(P,\eta)$ whose winding number about $\eta$ is exactly $n$.  Moreover, since $R'=0$ and $\R'(K)=0$,  both $R'$ and $\R'(K)$ have closures which are slice in $\Z[\frac{1}{n}]$-balls. Since fusing together components of a slice link yields a slice knot, the closures of $T$ and $T(K)$ are $0$ in $\C^{\frac{1}{n}}$. That is to say, the knot $\widetilde{P}$ as well as the knot $P(K)$ are zero in $\C^{\frac{1}{n}}$. Finally Corollary~\ref{cor:CFHH} applies to show $K=0$. This completes the proof of Proposition~\ref{prop:stringweak}.
\end{proof}

This completes the proof of Theorem~\ref{thm:stringlinkresult}.
\end{proof}

Proposition~\ref{prop:stringweak} can be extended to ordinary links. Suppose $L$ is an ordered oriented link of $m$ components embedded in the solid torus $ST$. Then, exactly as for pattern knots, we can define the link $L(K)$. We get an induced operator
$$
\cL:\C^*\to \C^*_m
$$
where the latter is the set of concordance classes of links in the category $*$. We can define the winding vector and winding number as for string links. Here we can go further and define a \textbf{strong winding number one link operator}  $(L,\eta)$ to be one for which there is some choice of parallels and fusions that take $(L,\eta)$ to a strong winding number one knot pattern $(P,\eta)$.

\begin{prop}\label{prop:linkweak} Suppose $\cL=(L,\eta)$ is a link operator with non-zero winding number  $n$ where $L$ is equivalent to the trivial link in $\C_m^*$ , then $\cL:\C^*\to \C_m^*$ is weakly injective for $*=\frac{1}{n}$. If $\cL$ has strong winding number one then $\cL:\C^*\to \C_m^*$ is weakly injective for $*=top$ and $*=ex$.
\end{prop}

\begin{proof} The proof is the same as that of Proposition~\ref{prop:stringweak} except that we also need Corollary~\ref{cor:CFHHmainstronger2}.
\end{proof}

\section{Further Questions}\label{sec:questions}

\noindent\textbf{Questions:}

\begin{itemize}
\item [1.]  Which, if any, ``non-trivial'' winding number zero operators are injective? 

Note that this includes the famous case of the Whitehead double operator. Proposition~\ref{prop:sameoperators}  ensures that \textit{some} winding number zero satellite operators will act trivially on $\C^*$. Namely suppose that $(\widetilde{P},\eta)$ is concordant to  a split link. Then $P(K)=P(U)$ for all $K$.  For example this occurs when $\widetilde{P}$ is a ribbon knot and $\eta$ is a linking circle to one of the ribbon bands. In this case the satellite operator is equivalent to the degenerate satellite operator where the pattern knot is disjoint from a meridional disk of $ST$. This is normally disallowed in the definition of a pattern so it can reasonably be considered to be a ``trivial'' satellite operator.  However it is not easy to place natural a priori conditions on the link $(\widetilde{P},\eta)$ that exclude this possibility. The definition of a ``robust operator'' introduced in ~\cite{CHL5}  gave one such set of conditions. Another was given in ~\cite{Franklin1}. Is Proposition~\ref{prop:sameoperators} the only source of non-injectivity in the smooth category?  

\item [2.]  When do distinct patterns give distinct operators?

Note that if the patterns $P$ and $Q$ give identical functions on $\C^*$ (or $\C$) then $\widetilde{P}=P(U)=Q(U)=\widetilde{Q}$ in $\C^*$ (or $\C$).  Hence if $\widetilde{P}\neq \widetilde{Q}$ then the operators are distinct.  Once again, Proposition~\ref{prop:sameoperators} must be taken into account, so that one ought to perhaps consider operators modulo concordance (of the associated links). 

In particular, how many distinct strong-winding number one operators are there?  Consider the special case that $\widetilde{P}$ is unknotted. We have seen that any two component link with linking number one and each component unknotted corresponds to such a strong-winding number one operator. If such a $2$-component link were concordant to the positive Hopf link (equal in $\C^*$) then Proposition~\ref{prop:sameoperators} shows that the resulting operator would be equal to that of the Hopf link, which is the identity operator. However, as mentioned in Section~\ref{sec:satellitesandwinding}, there are many such links (operators) that are not concordant to the Hopf link as evidenced by several recent papers ~\cite{ChaKimRubSt, FrPo, CFHH}.  Thus there appear to exist a large number of  distinct strong winding number $\pm 1$ operators wherein $\widetilde{P}$ is unknotted. Can it be proved that these are \textit{always} distinct?

B. Franklin has considered two component links, $(R,\eta)$ with $\eta$ unknotted, and proven that, even after \textit{fixing} the first component, $R$, that many (in fact in a precise sense almost all !) choices of $\eta$ lead to \textit{distinct} (winding-number zero) operators ~\cite{Franklin1}\cite[Section 5]{Franklin2}.

\item [3.]  When are  winding number one operators surjective?

S. Akbulut has conjectured that there exists a winding number $1$ operator $P$ for which $0$ is not in the image of 
$P:\C\to \C$ ~\cite[Problem 1.45]{Kirbyproblemlist}. By contrast, it is clear that winding number zero operators are not surjective since in this case, for example, $P(K)$ is the same as $\widetilde{P}$ in the algebraic knot concordance group. The image of the Whitehead double operator consists entirely of knots that are topologically slice.

\end{itemize}

\bibliographystyle{plain}
\bibliography{mybib4_30_2012}

\end{document}